\documentclass[12pt, reqno]{amsart}
\usepackage{amsmath, amsthm, amscd, amsfonts, amssymb, graphicx, color}
\usepackage[bookmarksnumbered, colorlinks, plainpages]{hyperref}
\hypersetup{colorlinks=true,linkcolor=red, anchorcolor=green, citecolor=cyan, urlcolor=red, filecolor=magenta, pdftoolbar=true}

\newtheorem{thm}[subsection]{Theorem}
\newtheorem{defn}[subsection]{Definition}
\newtheorem{lemma}[subsection]{Lemma}
\newtheorem{cor}[subsection]{Corollary}
\newtheorem{prop}[subsection]{Proposition}
\theoremstyle{definition}

\newtheorem{rmk}[subsection]{Remark}
\newtheorem{eg}[subsection]{Example}
\newtheorem{qn}[subsection]{Question}
\numberwithin{equation}{section}

\begin{document}

\title[ minimum  attaining Closed Operators ]{Absolutely minimum attaining Closed Operators  }

\author[S. H. Kulkarni,  G. Ramesh]{ S. H. Kulkarni$^1$ and  G. Ramesh $^2*$}

\address{$^{1}$ Department of Mathematics\\ I. I. T. Madras, Chennai\\Tamilnadu, India 600 036.}
\email{shk@iitm.ac.in}

\address{$^{2}$Department of Mathematics\\I. I. T. Hyderabad, Kandi(V)\\ Sangareddy, Telangana \\ India-502 285.}

\email{rameshg@iith.ac.in}

\subjclass[2000]{ 47A75, 47A05, 47A10, 47A15 }
\keywords{closed operator, minimum modulus, absolutely minimum attaining operator, invariant subspace, Lomonosov theorem, generalized inverse.}
\date{\today .
\newline \indent $^{*}$Corresponding author}
\begin{abstract}
We define and discuss properties of the class  of unbounded operators which  attain minimum modulus.
We establish a relationship between this class and the class of norm attaining bounded operators and compare the properties of both.
Also we  define absolutely minimum attaining operators (possibly unbounded) and characterize injective absolutely minimum attaining operators  as those with compact generalized inverse.
We give several consequences, one of those is that every such operator has a non trivial hyper invariant subspace.
\end{abstract}
\maketitle

\section{Introduction }
The class of norm attaining operators on Banach spaces is well studied by several authors in the literature. It is known that the class of norm attaining operators  is dense in the space of all bounded linear operators on a Hilbert space with respect to the operator norm \cite[Theorem 1]{enfloetal}.
For more details on norm attaining operators on Banach spaces, we refer to \cite{shkarin,acostaetal} and \cite{falcoetal} and references therein.

Every compact operator is norm attaining. In fact, restricted to any non zero closed subspace of a Hilbert space, it remains as compact and hence norm attaining. Motivated by this observation, Carvajal and Neves \cite{carvajalneves1} introduced a class of operators, called the absolutely norm attaining operators. Characterization of  such  operators on separable Hilbert space, in a particular case is given in \cite{rameshstructurethm} and
a complete characterization on arbitrary  Hilbert space is discussed in \cite{SP}. Many properties of these operators resemble the properties of compact operators.

It is a natural question to ask what happens if the norm is replaced by the minimum modulus. This leads to the definition of minimum attaining operators. Analogously, we can define absolutely minimum attaining operators. In a recent paper Carvajal and Neves \cite{carvajalneves2}, studied bounded operators between two different Hilbert spaces having such property.
 The structure of positive absolutely minimum attaining operators is  described in \cite{grs2}. This concept is also applicable to linear operators, that are not bounded.

In this article, we introduce the minimum attaining property for densely defined closed operators (possibly not bounded). We prove several characterizations  of such  operators.  We also prove the dual
relation between the norm attaining bounded operators and the minimum attaining closed operators. Finally, we introduce the absolutely minimum attaining
operators and prove a representation theorem for the injective absolutely minimum attaining operators. Furthermore, we observe that this class is exactly the  same as the class of densely defined closed operators whose generalized inverse is compact. Finally, we show that these operators possess a non trivial hyperinvariant subspace.

We organize the article as follows: In the second section we provide basic results which will be used  throughout the article. In the third section, we define minimum attaining property for densely defined closed operators and prove several characterizations. Some of the results in this section generalize the existing results of bounded operators and some of them are new. In the fourth section, we define absolutely minimum attaining operators and show that all such operators have a closed range. In particular, we show that an injective densely defined closed operator is  absolutely minimum  attaining if and only if its Moore-Penrose inverse is compact. Using this result, we deduce several consequences. One of the
important consequences is that every such operator has a non trivial hyper invariant subspace.

\section{Preliminaries}
In this section we introduce some basic notations, definitions and results that are  needed to prove our main results.

Throughout the article we consider infinite dimensional complex Hilbert spaces which will be denoted by $H, H_1,H_2$ etc. The inner product and the
induced norm are denoted  by  $\langle \cdot\rangle$ and
$||.||$, respectively. Let $T$ be a linear operator with domain $D(T)$, a subspace of $H_1$ and taking values in $H_2$. If $D(T)$ is dense in $H_1$, then $T$ is called a densely defined operator. The  graph $G(T)$ of $T$ is defined by  $G(T):={\{(Tx,x):x\in D(T)}\}\subseteq H_1\times H_2$. If $G(T)$ is
closed, then $T$ is called a closed operator.  Equivalently, $T$ is closed if and only if  $(x_n)$ is a sequence in  $D(T)$  such that  $x_n\rightarrow x\in H_1$ and $Tx_n\rightarrow y\in H_2$, then $x\in D(T)$ and $Tx=y$.

By the closed graph Theorem \cite{rud}, an everywhere defined
closed operator is bounded.  Hence the domain of an unbounded closed operator is a proper subspace of a Hilbert space.

The space of all bounded operators between $H_1$ and $H_2$ is
denoted by $\mathcal B(H_1,H_2)$ and the class of all closed operators between $H_1$ and $H_2$ is denoted by $\mathcal C(H_1,H_2)$. We write $\mathcal B(H,H)=\mathcal B(H)$  and $\mathcal C(H,H)=\mathcal C(H)$.

If $T\in \mathcal C(H_1,H_2)$, then  the null space and the range space of $T$ are denoted by $N(T)$ and $R(T)$ respectively and the space $C(T):=D(T)\cap N(T)^\bot$ is called the carrier of $T$.
In fact, $D(T)=N(T)\oplus^\bot C(T)$ \cite[page 340]{ben}.

For a densely defined operator, there exists a unique linear operator (in fact, a closed operator) $T^*:D(T^*)\rightarrow H_1$, with
\begin{equation*}
D(T^*):={\{y\in H_2: x\rightarrow \langle Tx,y\rangle \, \text{for all}\, x\in D(T)\,\text{is continuous}}\}\subseteq H_2
\end{equation*}
 satisfying $\langle Tx,y\rangle =\langle x,T^*y\rangle$ for all $x\in D(T)$ and $y\in D(T^*)$. it is to be noted that $T^*$ exists if and only if $T$ is densely defined.

If $S$ and $T$ are closed operators  with the property that $D(T)\subseteq D(S)$ and $Tx=Sx$ for all $x\in D(T)$, then $S$ is called the restriction of $T$ and $T$ is called an extension of $S$.  Furthermore, $S=T$ if and only if  $S\subseteq T$ and $T\subseteq S$.

If $S\in \mathcal B(H)$ and $T\in \mathcal C(H)$ is densely defined, the we say $S$ and $T$ are commuting if $ST\subseteq TS$. That is, $D(ST)\subseteq D(TS)$ and $STx=TSx$ for all $x\in D(ST)$.

A densely defined operator $T\in \mathcal C(H)$ is said to be normal if $T^*T=TT^*$, self-adjoint if $T=T^*$ and positive if $\langle Tx,x\rangle \geq 0$ for all $x\in D(T)$.

If $T$ is positive, then there exists a unique positive operator $S$ such that $T=S^2$. The operator $S$ is called the square root of $T$ and it is denoted by $S=T^\frac{1}{2}$.

 If $T\in \mathcal C(H_1,H_2)$ is densely defined, then the  operator  $|T|:=(T^*T)^\frac{1}{2}$ is called the modulus of $T$. There exists a unique partial isometry $V:H_1\rightarrow H_2$ with initial space $\overline{R(T^*)}$ and range $\overline{R(T)}$ such that $T=V|T|$.

 It can be verified that $D(|T|)=D(T)$ and $N(|T|)=N(T)$ and $\overline{R(|T|)}=\overline{R(T^*)}$.

 Let $T\in \mathcal C(H)$ be densely defined. The resolvent of $T$ is defined by
 \begin{equation*}
 \rho(T):={\{\lambda \in \mathbb C: T-\lambda I:D(T)\rightarrow H\; \text{is invertible and}\; (T-\lambda I)^{-1}\in \mathcal B(H)}\}
 \end{equation*}
and
\begin{align*}
\sigma(T):&=\mathbb C \setminus \rho(T)\\
\sigma_p(T):&={\{\lambda \in \mathbb C: T-\lambda I:D(T)\rightarrow H \; \text{is not one-to-one}}\},
\end{align*}
are called the  spectrum and the point spectrum of $T$, respectively.

%If $S,T\in \mathcal C(H_1,H_2)$, then $D(S+T)=D(S)\cap D(T)$ and $(S+T)(x)=Sx+Tx$ for all $x\in D(S)\cap D(T)$. If $R\in \mathcal C(H_2,H_3)$, then $RT$ is a linear operator with $D(RT)={\{x\in D(T): Tx\in D(R)}\}$ and $(RT)x=R(Tx)$ for all $x\in D(RT)$.

Let $T\in \mathcal C(H_1,H_2)$ be densely defined. A subspace $D$ of $D(T)$ is called a core for $T$ if for any $x\in D(T)$, there exists a sequence $(x_n)\subset D$ such that $\displaystyle \lim_{n\rightarrow \infty}x_n=x$ and
$\displaystyle \lim_{n\rightarrow \infty}Tx_n=Tx$. In other words, $D$ is dense in the graph norm, which is defined by $\||x\||:=\|x\|+\|Tx\|$ for all $x\in D(T)$. It is a well known fact that $D(T^*T)$ is a core for $T$ (see \cite[Proposition 3.18, page 47]{schmudgen} for details).

If $M$ is a closed subspace of a Hilbert space $H$, then $P_M$ denotes the orthogonal projection $P_M:H\rightarrow H$ with range $M$,  and $S_M:={\{x\in M:\|x\|=1}\}$ is the unit sphere of $M$.

We refer \cite{akhglazman,taylorlay,goldberg,rud,birmannsolomyak,schmudgen} for the above basics of unbounded operators.

Here we recall definition and properties of the Moore-Penrose inverse (or generalized inverse) of a densely defined closed operator that we need for our purpose. We refer \cite{ben} for more details on this topic.

Let $T\in \mathcal C(H_1,H_2)$ be densely defined. Then there exists
a unique densely defined operator $T^\dagger \in \mathcal
C(H_2,H_1)$ with domain $D(T^\dagger)=R(T)\oplus ^\bot R(T)^\bot$
and has the following properties:
\begin{enumerate}
\item $TT^\dagger y=P_{\overline{R(T)}}~y, ~\text{for all}~y\in D(T^\dagger)$

\item $T^\dagger Tx=P_{N(T)^\bot} ~x, ~\text{for all}~x\in D(T)$

\item $N(T^\dagger)=R(T)^\bot$.
\end{enumerate}
This unique operator $T^\dagger$ is called the \textit{Moore-Penrose inverse} of $T$.\\
The following property of $T^\dagger$ is also well known.

For every $y\in D(T^\dagger)$, let $$L(y):=\Big\{x\in D(T):
||Tx-y||\leq ||Tu-y||\quad \text{for all} \quad u\in D(T)\Big\}.$$
 Here any $u\in L(y)$ is called a \textit{least square solution} of the operator equation $Tx=y$. The vector  $x=T^\dagger y\in L(y),\,||T^\dagger y||\leq ||x||\quad \text{for all} \quad x\in L(y)$
 and it is called the  \textit{least square solution of minimal norm}.
 A different treatment of $T^\dagger$ is given in \cite[Pages 314, 318-320]{ben},
 where it is called ``\textit{the Maximal Tseng generalized Inverse}".

Here we recall some properties of $T^{\dagger}$ that we will be using very frequently.
\begin{thm}\cite[Page 320]{ben}\label{propertiesmpi}
Let $T\in \mathcal C(H_1,H_2)$ be densely defined. Then
\begin{enumerate}
\item $D(T^\dagger)=R(T)\oplus^\bot R(T)^\bot,\;
N(T^\dagger)=R(T)^\bot=N(T^*)$
\item $R(T^\dagger)=C(T)$
\item $T^\dagger$ is densely defined and $ T^\dagger \in \mathcal C(H_2,H_1)$
\item $T^\dagger$ is continuous if and only $R(T)$ is closed
\item $T^{\dagger \dagger}=T$
\item $T^{* \dagger}=T^{\dagger *}$
\item $N(T^{* \dagger})=N(T)$
\item $T^*T$ and $T^\dagger T^{* \dagger}$ are positive and $(T^*T)^\dagger =T^\dagger T^{*
\dagger}$
\item $TT^*$ and $ T^{* \dagger}T^\dagger$ are positive and $(TT^*)^\dagger= T^{*
\dagger}T^\dagger$.
\end{enumerate}
\end{thm}
%\end{defn}
 %\begin{defn}\cite[Page 267]{kato}
 %Let $T\in \mathcal C(H)$ be densely defined. Then
 %the \textit{numerical range} of $T$ is defined by
%\begin{equation*}
%W(T):=\Big\{\langle Tx,x\rangle: x\in S_{D(T)}\Big\}.
%\end{equation*}
%If $T=T^*$, then the quantities
%\begin{align*}
%m_T:=\inf \Big\{\langle Tx,x\rangle: x\in S_{D(T)}\Big\}; \quad
%M_T:=\sup \Big\{\langle Tx,x\rangle : x\in S_{D(T)}\Big\}
%\end{align*}
%are called the \text{lower} and \text{upper} numerical bounds of $T$, respectively.
 %\end{defn}
%The following Proposition is proved in \cite[Chapter 10]{lance} for regular (unbounded) operators between Hilbert $C^*$-modules, which is obviously true for densely defined closed operators in a Hilbert space.
%\begin{prop}\cite[Lemma 5.8]{schmudgen} Let $T\in \mathcal C(H)$ be densely defined.  Let $Q_T:=(I+T^*T)^{-\frac{1}{2}}$ and $Z_T:=TQ_T$. Then
%\begin{enumerate}
%\item $Q_T\in \mathcal B(H)$ and $0\leq Q_T\leq I$
%\item $R(Q_T)=D(T)$
%\item $(Z_T)^*=F_{T^*}$
%\item $\|Z_T\|<1$ if and only if $T\in \mathcal B(H_1,H_2)$
%\item $T=Z_T(I-Z_T^*Z_T)^{-\frac{1}{2}}$
%\item $Q_T=(I-Z_T^*Z_T)^{\frac{1}{2}}$.
%\end{enumerate}
%The operator $Z_T$ is called the bounded transform of $T$ or the $z$-transform of $T$.
%\end{prop}

\section{Minimum attaining Operators}

In this section first we discuss some important properties of minimum attaining operators. These operators for the bounded case was discussed in \cite{carvajalneves2} and  the unbounded case in \cite{shkgrminattaining}.
It is proved that this class is dense in the class of densely defined closed operators with respect to the gap topology.
\begin{defn}\cite{ben, goldberg}\label{minmmodulus}
 Let $T\in \mathcal C(H_1,H_2)$ be densely defined. Then
 \begin{align*}
 m(T)&:=\inf{\{\|Tx\|: x\in S_{D(T)}}\}\\
 \gamma(T)&:=\inf{\{\|Tx\|:x\in S_{C(T)}}\},
 \end{align*}
 are  called the minimum modulus and the reduced minimum modulus of $T$, respectively. The operator $T$ is said to be bounded  below if and only if $m(T)>0$.
\end{defn}

\begin{rmk}
If $T\in \mathcal C(H_1,H_2)$ is densely defined, then
\begin{itemize}
\item[(a)] By definition, we have $m(T)\leq \gamma(T)$. More over, if $T$ is one-to-one, $m(T)=\gamma(T)$ since $D(T)=C(T)$
\item[(b)] $m(T)>0$ if and only if $R(T)$ is closed and $T$ is one-to-one
\item[(c)] Since $D(T)=D(|T|)$ and $\|Tx\|=\||T|x\|$ for all $x\in D(T)$, we can conclude that $m(T)=m(|T|)$ and $\gamma(T)=\gamma(|T|)$.
    \end{itemize}
\end{rmk}

\begin{rmk}\label{reciprocalinversenorm}
 If $T\in \mathcal C(H)$ is densely defined and $R(T)$ is closed, then $\gamma(T)=\frac{1}{\|T^{\dagger}\|}$.

\end{rmk}

 Recall that $T\in\mathcal B(H_1,H_2)$ is said to be norm attaining if there exists  $x_0\in S_{H_1}$ such that
$\|Tx_0\|=\|T\|$. We denote the class of all norm attaining operators between $H_1$ and $H_2$ by $\mathcal N(H_1,H_2)$. In case $H_1=H_2=H$, we denote this by $\mathcal N(H)$. In a similar way, we can define operators that attain  minimum modulus. The class of bounded operators that attain  minimum modulus is defined and several characterizations are proved  in \cite{carvajalneves1}. Here we discuss the same for unbounded operators.
\begin{defn}
Let $T\in \mathcal C(H_1,H_2)$ be densely defined. If there exists  $x_0\in S_{D(T)}$ such that $\|Tx_0\|=m(T)$, then we call $T$ to be minimum attaining.
\end{defn}

We write $$\mathcal{M}_{c}(H_1,H_2)={\{T\in \mathcal C(H_1,H_2): T \; \text{is densely defined and minimum attaining}}\}$$ and $\mathcal M_{c}(H,H)=\mathcal M_{c}(H)$.

\begin{rmk}
\begin{enumerate}
 Let $T\in \mathcal{C}(H_1,H_2)$ be densely defined.
\item  If $T$ is not one-to-one, then $m(T)=0$ and there exists a $x_0\in S_{N(T)}$ such that $Tx_0=0$.  Hence $T\in \mathcal M_{c}(H_1,H_2)$.
\item If $T$ is  one-to-one and $R(T)$ is not closed, then $m(T)=0$. But there does not exists $x_0\in D(T)$ such that $\|Tx_0\|=0$, since $T$ is one-to-one. Thus $T\notin \mathcal M_{c}(H_1,H_2)$.
\end{enumerate}
From the above two observations it is apparent that the injectivity of the operator plays an important role in the minimum attaining property.
\end{rmk}

%\begin{cor} \label{minmsquarerule}
%Let $T\in \mathcal C(H_1,H_2)$ be densely defined. Then
%$m(T^*T)=m(T)^2$.
%\end{cor}
%\begin{proof}
%Since $T^*T\geq 0$ (see \cite{rud} for details), by Proposition \ref{equalityofmodulus}, we can conclude that
%\begin{align*}
%m(T^*T)=m_{T^*T}&=\inf{\{\langle T^*Tx,x\rangle : x\in D(T^*T),\; \|x\|=1}\}\\
%                                 & =\inf{\{\|Tx\|^2: x\in D(T^*T)}\}\\
%                                 &\geq \inf{\{ \|Tx\|^2: x\in D(T)}\}  \;\; \text{since}\; D(T^*T)\subseteq D(T)
%                                 \end{align*}
%\end{proof}

First, we establish some results related to the minimum modulus of a densely defined closed operator, which are useful in discussing the minimum attaining property.
\begin{prop}\label{fnlcalulusformods}
Let $T\in \mathcal C(H)$ be densely defined and normal. Then
\begin{enumerate}
\item \label{mmspectralformula}  $m(T)=d(0,\sigma(T))$
\item \label{polymod}$m(T^n)=m(T)^n$.
\end{enumerate}
\end{prop}
\begin{proof}
If $m(T)=0$, then $T$ is not invertible, so $0\in \sigma(T)$ and  $d(0,\sigma(T))=0$. If $m(T)>0$, then $T^{-1}\in \mathcal B(H)$. In this case, $m(T)=\gamma(T)=\frac{1}{\|T^{-1}\|}$.
Therefore,
\begin{align*}
\frac{1}{\|T^{-1}\|}&=\dfrac{1}{\sup{\{\mu:\mu\in \sigma(T^{-1})}\}}\\
                    &=\dfrac{1}{\sup{\{\frac{1}{\lambda}:\lambda \in \sigma(T)}\}}\\
                    &=\inf{\{\lambda:\lambda \in \sigma(T)}\}\\
                    &=d(0,\sigma(T)). \qedhere
                    \end{align*}

Proof of (\ref{polymod}):  It is easy to verify that $T^n$ is normal. Hence, by (\ref{mmspectralformula}) and the spectral mapping theorem we can conclude that
\begin{align*}
m(T^n)&=\inf{\{|\mu|: \mu \in \sigma(T^n)}\}\\
&=\inf{\{|\lambda^n|:\lambda \in \sigma(T)}\}\\
                                           &=\inf{\{|\lambda|^n:\lambda \in \sigma(T)}\}\\
                                           &=m(T)^n.
                                           \end{align*}
\end{proof}
\begin{cor}\label{squarerootminmod}
If $T\in \mathcal C(H_1,H_2)$ is densely defined, then
\begin{enumerate}
\item $m(T)=d(0,\sigma(|T|))$
\item $m(T^*T)=m(T)^2$.
\end{enumerate}
\end{cor}
\begin{proof}
  We have $m(T)=m(|T|)=d(0,\sigma(|T|))$, by  (\ref{mmspectralformula}) of Proposition \ref{fnlcalulusformods}. Also, $m(T^*T)=m(|T|^2)=d(0,\sigma(|T|^2))=d(0,\sigma(|T|))^2=m(|T|)^2=m(T)^2$. Here we have used both (\ref{mmspectralformula})  and (\ref{polymod}) of Proposition \ref{fnlcalulusformods} to get the conclusion.
  \end{proof}
  \begin{prop}\label{equalityofmodulus}
Let $T\in \mathcal C(H)$ be densely defined and positive. Then
\begin{equation*}
m(T)=\inf \big\{\langle Tx,x\rangle :x\in S_{D(T)}\big\}=m_T.
\end{equation*}
\end{prop}
\begin{proof}
 First note that $D(T)\subseteq D(T^{\frac{1}{2}})$. Next,
\begin{align*}
m_T=\inf \big\{\langle Tx,x\rangle: x\in S_{D(T)} \big\}
      & =\inf \big\{\langle T^{\frac{1}{2}}x,T^{\frac{1}{2}}x\rangle: x\in S_{D(T)} \big\}\\
                                                                                      &\geq \inf \big\{\|T^{\frac{1}{2}}x\|^2: x\in D(T^{\frac{1}{2}}) \big\}\\
                                                                                      & =m(T^{\frac{1}{2}})^2.
\end{align*}
But $m(T^{\frac{1}{2}})=\inf{\{\lambda: \lambda \in \sigma(T^{\frac{1}{2}})}\}$ by Corollary  \ref{squarerootminmod}.  As $\sigma(T)={\{\lambda^2: \lambda \in \sigma(T^{\frac{1}{2}})}\}$,  we have  that $m(T^{\frac{1}{2}})^2=m(T)$ and hence  $m_T\geq m(T)$.

On the other hand, we have
\begin{align*}
m_T &\leq \langle Tx,x\rangle \; \text{for all}\;  x\in S_{D(T)}\\
      &= \langle T^{\frac{1}{2}}x,T^{\frac{1}{2}}x\rangle \; \text{for all}\;  x\in S_{D(T)}  \\
       &=\|T^{\frac{1}{2}}x\|^2\;  \text{for all}\;  x\in S_{D(T)}.
\end{align*}
Next, we claim that the above inequality holds for all $x\in D(T^{\frac{1}{2}})$. To this end, let $x\in D(T^{\frac{1}{2}})$. Since, $D(T)$ is a core for $T^{\frac{1}{2}}$, there exists a sequence $(x_n)\subset D(T)$ such that $\displaystyle \lim_{n\rightarrow \infty}x_n\rightarrow x$ and $\displaystyle \lim_{n\rightarrow \infty}T^{\frac{1}{2}}x_n=T^{\frac{1}{2}}x$. Hence
$\|T^{\frac{1}{2}}x\|^2=\displaystyle \lim_{n\rightarrow \infty}\|T^{\frac{1}{2}}x_n\|^2 \geq m_T$. As this is true for all $x\in D(T^{\frac{1}{2}})$, it follows that $m(T)\geq m_T$.

By the above two observations the conclusion follows.
\end{proof}

\begin{prop}\label{adjointminmattaining}
Let $T\in \mathcal C(H_1,H_2)$ be densely defined and $m(T)=m(T^*)$.  Also, assume that $R(T)$ is closed. Then $T\in \mathcal M_{c}(H_1,H_2)$ if and only if $T^*\in \mathcal M_{c}(H_2,H_1)$.
\end{prop}
\begin{proof}
Clearly, if $m(T)=m(T^*)=0$, since $R(T)$ closed,  both $T$ and $T^*$ are not one-to-one. Hence both are minimum attaining. Now assume that $m(T)>0$. It is sufficient to prove one implication, since $T^{**}=T$ and $m(T)=m(T^*)$. By Proposition \ref{equivalentwithgramoperator}, $T\in \mathcal M_{c}(H_1,H_2)$ if and only if there exists a $x_0\in S_{D(|T|)}$ such that $|T|x_0=m(T)x_0$. That is $T^*Tx_0=m(T)|T|x_0$. Hence
\begin{equation*}\frac{\|T^*Tx_0\|}{\|Tx_0\|}=m(T)\frac{\||T|x_0\|}{\|Tx_0\|}=m(T),
\end{equation*}
proving $T^*\in \mathcal M_{c}(H_2,H_1)$.
\end{proof}
\begin{rmk}
Let $T\in \mathcal C(H)$ be densely defined and  normal. Then $D(T)=D(T^*)$ and $\|Tx\|=\|T^*x\|$ for all $x\in D(T)$. Hence $T\in \mathcal M_{c}(H)$ if and only if $T^*\in \mathcal M_{c}(H)$. Clearly, in this case $m(T)=m(T^*)$. Note that in this case we don't have to assume that the range of $T$  to be closed.
%\item In general, it is not true that $T\in \mathcal M_{c}(H)$ imply that $T^*\in \mathcal M_{c}(H)$.

%To see this, consider the example:  Let $H=\ell^2$ and $R$ be the right shift operator on $H$ and $D$ be the diagonal operator, defined by $De_n=\frac{1}{n}e_n$, where ${\{e_n:n\in \mathbb N}\}$ is the standard orthonormal basis for $H$. Let $S=R(I-D)$. Then $S^*S=(I-D)^2 \notin \mathcal M_{c}(H)$, by Proposition \ref{equivalentwithgramoperator} and hence $S\notin \mathcal M_{c}(H)$. But,
%\begin{equation*}
%SS^*(x)=(0,0, \frac{1}{4}x_3,\frac{4}{9}x_4,\dots,),\; \text{for all}\; (x_n)\in H.
%\end{equation*}
%Clearly, $SS^*e_1=0=m(SS^*)$. Hence  by Proposition, \ref{equivalentwithgramoperator}, $SS^*\in \mathcal M_{c}(H)$. Hence $S^*\in \mathcal M_{c}(H)$.
\end{rmk}

We recall that if $T\in \mathcal C(H_1,H_2)$ is densely defined, then the numerical range $W(T)$ of $T$ is defined by $W(T)={\{\langle Tx,x\rangle: x\in S_{D(T)}}\}$.
\begin{prop}\label{eigenvalueforpositive}
If $T\in \mathcal C(H)$  is  positive, then the following are equivalent;
\begin{enumerate}
\item \label{cond1} $T\in \mathcal M_{c}(H)$
\item \label{cond2}$m(T)\in \sigma_p(T)$
\item \label{extremepointpositive}$m(T)$ is an extreme point of $W(T)$.
\end{enumerate}
\end{prop}
\begin{proof}
Proof of $ (\ref{cond1})\Rightarrow (\ref{cond2}):$  Choose  $x_0\in S_{D(T)}$ such that $\|Tx_0\|=m(T)$. Since, $T-m(T)I$ is positive, and by the Cauchy-Scwarz inequality,
we get that
\begin{equation*}
 m(T)\leq \langle Tx_0,x_0\rangle \leq \|Tx_0\|=m(T),
\end{equation*}
or $m(T)=\langle Tx_0,x_0\rangle$. Therefore,
\begin{align*}
\|Tx_0-m(T)x_0\|^2&=\|Tx_0\|^2+m(T)^2-2m(T)\langle Tx_0,x_0\rangle \\
                  &=2m(T)^2-2m(T)^2\\
                  &=0.
\end{align*}
That is, $Tx_0=m(T)x_0$. Clearly, if $m(T)\in \sigma_p(T)$, then $T\in \mathcal M_{c}(H)$.

%\begin{equation}\label{minmeigenvalueeqn}
%\langle Tx_0,Tx_0\rangle=m(T)^2\langle x_0,x_0\rangle.
% \end{equation}
%As $(I+T^2)^{\frac{1}{2}}$  is bijective and has a bounded inverse, choose $y_0\in H$ such that $x_0=(I+T^2)^{-\frac{1}{2}}y_0$. Then Equation (\ref{minmeigenvalueeqn}) can be written as:
%\begin{align*}
%\langle T(I+T^2)^{-\frac{1}{2}}y_0, T(I+T^2)^{-\frac{1}{2}}y_0\rangle &=m(T)^2\langle (I+T^2)^{-\frac{1}{2}}y_0,(I+T^2)^{-\frac{1}{2}}y_0\rangle\\
%\langle T(I+T^2)^{-\frac{1}{2}}T(I+T^2)^{-\frac{1}{2}}y_0, y_0\rangle &=m(T)^2 \langle (I+T^2)^{-\frac{1}{2}}y_0,(I+T^2)^{-\frac{1}{2}}y_0\rangle\\
%\langle T^2(I+T^2)^{-1}y_0, y_0\rangle &=m(T)^2 \langle (I+T^2)^{-1}y_0,y_0\rangle.
%\end{align*}
%Thus
%\begin{equation}\label{eigenvalueeqnminmod}
%\langle \big(T^2-m(T)^2\big)(I+T^2)^{-1}y_0, y_0\rangle =0.
%\end{equation}
%           Since $ \big(T^2-m(T)^2 I\big)(I+T^2)^{-1}\geq 0$, it follows that $ \big(T^2-m(T)^2 I\big)(I+T^2)^{-1}y_0=0$.  That is $(T+m(T)I)(T-m(T)I)(I+T^2)^{-1}y_0=0$.  If $m(T)>0$, then  $(T+m(T)I)^{-1}\in \mathcal B(H)$, and hence  $T(I+T^2)^{-1}y_0=m(T)(I+T^2)^{-1}y_0$. Hence $z_0=(I+T^2)^{-1}y_0$ is an eigenvector of $T$ corresponding to $m(T)$.

%           In case, if $m(T)=0$, then  Equation \ref{eigenvalueeqnminmod}, takes the form $T^2(I+T^2)^{-1}y_0=0$. That is $y_0\in N(Z_T^2)=N(Z_T)$. But, by \cite[Proposition 4.2(2)]{shkgr3}, $N(Z_T)=N(T)$. Hence $y_0\in N(T)$.   In this case also, $z_0$ is the eigenvector corresponding to $m(T)$. Hence the conclusion follows.

Proof of $(\ref{cond2})\Rightarrow  (\ref{extremepointpositive}):$ Let $x_0\in S_{D(T)}$ be such that $Tx_0=m(T)x_0$. Then $m(T)=\langle Tx_0,x_0\rangle \in W(T)$. Since, $m(T)=m_T$ by Proposition \ref{equalityofmodulus}, the conclusion follows. The other way implication follows by  the main theorem of \cite{bernau1}.
\end{proof}
Using Proposition \ref{eigenvalueforpositive}, we can prove the following.
\begin{prop}\label{equivalentwithsquarerroot}
Let $T\in \mathcal C(H)$ be densely defined and positive. Then $T\in \mathcal M_{c}(H)$ if and only if  $T^{\frac{1}{2}}\in \mathcal M_{c}(H)$.
\end{prop}
\begin{proof}
If $T^\frac{1}{2}\in \mathcal M_{c}(H)$, then $m(T^\frac{1}{2})\in \sigma_p(T^{\frac{1}{2}})$, which implies that $m(T)\in \sigma_p(T)$. By Proposition \ref{eigenvalueforpositive}, $T\in \mathcal M_{c}(H)$.

Conversely, if $T\in \mathcal M_{c}(H)$, then $m(T)\in \sigma_p(T)$, by Proposition \ref{eigenvalueforpositive}.  If $m(T)=0$, then $m(T^{\frac{1}{2}})=0$ and hence $T^{\frac{1}{2}}\in \mathcal M_{c}(H)$. Next, assume that $m(T)>0$. Then $m(T^{\frac{1}{2}})>0$ and
\begin{equation*}
T-m(T)I=\big(T^\frac{1}{2}+m(T)^\frac{1}{2}I\big)\big( T^\frac{1}{2}-m(T)^\frac{1}{2}I\big).
\end{equation*}
As $T^\frac{1}{2}+m(T)^\frac{1}{2}I$ has a bounded inverse, we have that $T^\frac{1}{2}-m(T)^\frac{1}{2}I$ is not one-to-one. Hence $m(T^\frac{1}{2})\in \sigma_p(T^\frac{1}{2})$. The conclusion  follows by Proposition \ref{eigenvalueforpositive}.
\end{proof}
\begin{thm}\label{equivalentwithgramoperator}
Let $T\in \mathcal C(H_1,H_2)$ be densely defined.  Then the following
statements are equivalent:
\begin{enumerate}
\item \label{main} $T\in \mathcal M_{c}(H_1,H_2)$
\item \label{mainmod}$|T|\in \mathcal M(H_1)$
\item \label{maingram}$T^*T\in \mathcal M(H_1)$.
\end{enumerate}
\end{thm}
\begin{proof}
The equivalence of (\ref{main}) and (\ref{mainmod}) follows by the observation that $D(T)=D(|T|)$ and $\|Tx\|=\||T|x\|$ for all $x\in D(T)$. The equivalence of (\ref{mainmod}) and (\ref{maingram}) follows by the fact that $T^*T=|T|^2$ and Proposition \ref{equivalentwithsquarerroot}.
\end{proof}

\begin{eg}
  Let $D={\{(x_n)\in \ell^2: \displaystyle \sum_{n=1}^{\infty}n^2|x_n|^2<\infty}\}$. Define $T:D\rightarrow \ell^2$ by
  \begin{equation*}
    T((x_1,x_2,x_3,\dots,))=((0,x_1,2x_2,3x_3,\dots)),\; \text{for all}\; (x_n)\in D.
  \end{equation*}
  Clearly, $T$ is densely defined closed operator. Note that $T^*T(x_n)=(n^2x_n)$ for all $(x_n)\in D(T^*T)$. It can be easily calculated that $\sigma(T^*T)=\sigma_{p}(T^*T)={\{n^2:n\in \mathbb N}\}$. Hence $m(T^*T)=1$ and $T^*T\in \mathcal M(\ell^2)$. By Theorem \ref{equivalentwithgramoperator}, we can conclude that $T\in \mathcal M(\ell^2)$ and by Corollary \ref{squarerootminmod} $m(T)=1$.
\end{eg}
\begin{prop}
Let $T$ be densely defined and positive. Then
 $T\in \mathcal M_{c}(H)$ if and only if $T^n\in \mathcal M_{c}(H)$ for each $n\geq 1$.
\end{prop}
\begin{proof}
Let $T\in \mathcal M_{c}(H)$. Then by Proposition \ref{eigenvalueforpositive}, there exists $x_0\in D(T)$ such that $Tx_0=m(T)x_0$. Observe that $x_0\in D(T^2)$. This implies that $T^2x_0=m(T)^2x_0=m(T^2)x_0$, by Proposition \ref{fnlcalulusformods}. That is $x_0\in D(T^4)\subseteq D(T^3)$. With this, we have $T^3x_0=m(T)^3x_0$. By the induction argument we can show that $T^nx_0=m(T)^nx_0$. By Proposition \ref{eigenvalueforpositive}, it follows that $T^n\in \mathcal M_{c}(H)$.

To prove the converse, assume that $n>1$ and  $T^n\in \mathcal M_{c}(H)$. Choose $x_0\in S_{D(T)}$ such that $T^nx_0=m(T^n)x_0$. As $m(T^n)=m(T)^n$, if $m(T^n)=0$, then $m(T)=0$. In this case $x_o\in N(T^n)$. That is $T^{n-2}x_0\in N(T^2)=N(T)$. Hence $x_0\in N(T^{n-1})$. Proceeding in this we can conclude that $x_0\in N(T)$, proving $T\in \mathcal M_{c}(H)$.

Next assume that $m(T)>0$. Since $T$ is positive, $T^{-1}\in \mathcal B(H)$. Hence $Tx_0=m(T)x_0$ implies that
\begin{equation*}
T^{n-1}x_0=T^{-1}T^{n}x_0=m(T)^nT^{-1}x_0 =m(T)^n\frac{x_0}{m(T)}=m(T)^{n-1}x_0.
\end{equation*}
By proceeding in this way, we can conclude that $Tx_0=m(T)x_0$. Hence $T\in \mathcal M_{c}(H)$.
\end{proof}
\begin{prop}\label{minmattainingclosedrange}
Let $T\in \mathcal M_{c}(H_1,H_2)$ be one-to-one. Then $R(T)$ is closed.
\end{prop}
\begin{proof}
If $R(T)$ is not closed, then $m(T)=0$. Since $T\in \mathcal M_{c}(H_1,H_2)$, there exists $x_0\in S_{D(T)}$ such that $\|Tx_0\|=0$, but contradicts $T$ to be one-to-one. Thus $R(T)$ is closed.
\end{proof}
\begin{rmk}
The condition one-one ness  is not necessary in Proposition \ref{minmattainingclosedrange}.
For example, let $P$ be a bounded orthogonal projection.
Then $R(P)$ is closed and it is minimum attaining but not one-to-one.
\end{rmk}
\begin{cor}
Let $T\in \mathcal M_{c}(H_1,H_2)$. Then $T$ is one-to-one if and only $T$ bounded below.
\end{cor}

Next, we will establish a relation between the minimum attaining property of the operator and the norm attaining property of its generalized inverse. First we prove a few results needed for this purpose.
\begin{prop}\label{mprelations}
Let $T\in \mathcal C(H_1,H_2)$ be densely defined. Then
\begin{enumerate}
\item\label{mprelations1} $|T^\dagger|=|T^*|^\dagger$
\item\label{mprelations2} $|(T^\dagger)^*|=|T|^\dagger$.
\end{enumerate}
\end{prop}
\begin{proof}
Proof of (\ref{mprelations1}): By definition of $|T^\dagger|$, and by Theorem \ref{propertiesmpi},
\begin{align*}
|T^\dagger|=\big((T^\dagger)^*T^\dagger\big)^\frac{1}{2}
           =\big((TT^*)^\dagger\big)^\frac{1}{2}
           =\big((|T^*|^2)^\dagger\big)^\frac{1}{2}
           =\big((|T^*|)^\dagger (|T^*|)^\dagger\big)^\frac{1}{2}
           &=|T^*|^\dagger.
\end{align*}
The proof of (\ref{mprelations2}) can be obtained by  replacing  $T$ by $T^*$ in (\ref{mprelations1}) and observing  that $(T^*)^{\dagger}=(T^*)^{\dagger}$ and $(T^*)^*=T$.
\end{proof}

\begin{thm}\label{equivalencegeninvcondition}
Let $T\in \mathcal C(H_1,H_2)$ be densely defined and one-to-one. Then the following statements are equivalent;
\begin{enumerate}

%\item \label{rmmattaining} $T$ attains its reduced minimum. That is, there exists $x_{0}\in S_{C(T)}$ such that $\|Tx_0\|=\gamma(T)$
\item \label{minmattaining}$T\in \mathcal M_{c}(H_1,H_2)$
\item \label{geninvchar} $R(T)$ is closed and $T^{\dagger}\in \mathcal{N}(H_2,H_1)$.
\end{enumerate}
\end{thm}
\begin{proof}
First assume that $T\in \mathcal M_{c}(H_1,H_2)$. Note that $R(T)$ is closed, by Proposition \ref{minmattainingclosedrange}. As $T$ is one-to-one, $m(T)>0$.  Choose $x_0\in D(T)$ such that $|T|x_0=m(T)x_0$. Hence
\begin{equation*}
m(T)|T|^{-1}x_0=|T|^{-1}|T|x_0=x_0.
\end{equation*}
So
\begin{equation}\label{eqcondeq1}
|T|^{-1}x_0=\frac{1}{m(|T|)}x_0.
\end{equation}
By Proposition \ref{mprelations}, we have $|T|^{-1}=|(T^{\dagger})^*|=(T^*)^{\dagger}$. Since $m(|T|)=m(T)=\dfrac{1}{\|T^{\dagger}\|}=\dfrac{1}{\||T^{\dagger}|\|}$, Equation \ref{eqcondeq1}, takes the form $|T^{\dagger}|x_0=\||T^{\dagger}|\|x_0$. Hence the conclusion follows.

To prove the  other implication, let $T^{\dagger}\in \mathcal N(H_2,H_1)$. Then clearly, $R(T)$ is closed. By \cite[Proposition 2.5]{carvajalneves1}, $S:=(T^{\dagger})^*\in \mathcal N(H_1, H_2)$. Hence by Proposition \ref{mprelations}, we have $|S|=|T|^{\dagger}$. Since $|T|^{\dagger}$ is positive and norm attaining, there exists $x_0\in S_{H_1}$ such that
\begin{equation}\label{modulusofmpi}
|S|x_0=\|S\|x_0=\|T^{\dagger}\|x_0=\frac{1}{m(T)}x_0.
\end{equation}
Note that $x_0\in R(|T|^{\dagger})=C(|T|)=C(T)\subseteq N(T)^{\bot}$. Premultiplying Equation \ref{modulusofmpi} by $|T|$ and noting that $R(|T|)=N(T)^{\bot}$, we have
$|T|x_0=m(T)x_0$, concluding $|T|\in \mathcal M(H_1)$. Hence $T\in \mathcal M_{c}(H_1,H_2)$ by Theorem \ref{equivalentwithgramoperator}.
\end{proof}
%\begin{prop}
%Let $T\in \mathcal C(H_1,H_2)$ be densely defined. Then
%\begin{enumerate}
%\item If $T\in \mathcal M_{c}(H_1,H_2)$ with $m(T)=m(T^*)$. Then $T^\dagger \in \mathcal N(H_2,H_1)$
%\item If $\sigma(T^*T)=\sigma(TT^*)$, then $T\in \mathcal M_{c}(H_1,H_2)$ if and only if $T^*\in M(H_2,H1)$.
%\item if $H_1=H_2=H$ and $T$ is normal. Then $T\in \mathcal M_{c}(H)$ if and only $T^*\in \mathcal M_{c}(H)$.
%\end{enumerate}
%\end{prop}
%\begin{rmk}
  %Recall that $T\in \mathcal C(H_1,H_2)$ with a dense domain is said to be \textit{reduced minimum attaining} if there exist $x_0\in S_{C(T)}$ such that $\|Tx_0\|=\gamma(T)$.  The equivalence of (\ref{rmmattaining}) and (\ref{geninvchar}) by dropping the injectivity of $T$ in Theorem \ref{equivalencegeninvcondition} can be proved using a different method. We refer \cite{shkgrrmm} for more details on this topic.
%\end{rmk}

Next, we show that minimum attaining property of a closed densely defined operator is related  to the minimum attaining property of the corresponding bounded transform.
If $T\in \mathcal C(H_1,H_2)$ is densely defined, then the operator $Z_{T}:=T(I+T^*T)^{-\frac{1}{2}}$ is called the bounded transform of $T$. More over, $T=Z_T(I-Z_{T}^*Z_{T})^{-\frac{1}{2}}$. We refer \cite[section 7.3, page 142]{schmudgen} for more details about these operators.
\begin{prop}\label{minmattainingboundedtransform}
Let $T\in \mathcal C(H_1,H_2)$ be densely defined. Then
\begin{enumerate}
\item\label{minmodbddtransform} $m(Z_T)=\dfrac{m(T)}{\sqrt{1+m(T)^2}}$
\item \label{minmodinvbddtransform}$m(T)^2=\dfrac{m(Z_T)^2}{1-m(Z_T)^2}$
\item\label{bddtransformreln} $T\in \mathcal M_{c}(H_1,H_2)$ if and only if $Z_T\in \mathcal M_{c}(H_1,H_2)$.
\end{enumerate}
\end{prop}
\begin{proof}
Proof of (\ref{minmodbddtransform}): We have that
\begin{equation*}
m(Z_T)^2=m(Z_T^*Z_T)=m(I-(I+T^*T)^{-1}) =1-\|(I+T^*T)^{-1}\|=1-\dfrac{1}{m(I+T^*T)},
\end{equation*}
by Remark  \ref{reciprocalinversenorm}. Hence
\begin{align*}
                    m(Z_T)^2&=1-\dfrac{1}{1+m(T)^2}\\
                    &=\dfrac{m(T)^2}{1+m(T)^2}.
\end{align*}

Proof of (\ref{minmodinvbddtransform}):
Note that $T^*T=Z_T^*Z_T(I-Z_T^*Z_T)^{-1}=(I-Z_T^*Z_T)^{-1}-I$. Thus
\begin{equation*}
m(T)^2=m(T^*T)=m(I-Z_T^*Z_T)^{-1}-1=\dfrac{1}{\|I-Z_T^*Z_T\|}-1=\dfrac{1}{1-m(Z_T^*Z_T)}-1,
\end{equation*}
by Remark \ref{reciprocalinversenorm}.
Hence  \begin{equation*}
m(T)^2 =\dfrac{1}{1-m(Z_T)^2}-1=\dfrac{m(Z_T)^2}{1-m(Z_T)^2}.
\end{equation*}
Proof of (\ref{bddtransformreln}):
In view of Theorem \ref{equivalentwithgramoperator}, it is enough to prove $T^*T\in \mathcal M(H_1)$ if and only if $Z_T^*Z_T\in \mathcal M(H_1)$. We know by (\ref{cond2}) of Proposition \ref{eigenvalueforpositive}, that $T^*T\in \mathcal M(H_1)$ if and only if $m(T)^2\in \sigma_p(T^*T)$. Since, $Z_T^*Z_T=T^*T(I+T^*T)^{-1}$ and $m(Z_T^*Z_T)=\frac{m(T)^2}{1+m(T)^2}$, it can be verified that $m(T)^2\in \sigma_p(T^*T)$ if and only if $m(Z_T^*Z_T)\in \sigma_p(Z_T^*Z_T)$.
\end{proof}

\section{Absolutely minimum attaining operators}
In this section, we define absolutely minimum attaining operators and describe the structure of such operators.
%\begin{lemma}\label{decompositionwithclosedsubsp}
%Let $T\in \mathcal C(H_1,H_2)$ be densely defined and $M$ be a closed subspace of $D(T)$. Then
%\begin{enumerate}
%\item \label{denseness} if $D(T)\cap M$ is non-zero, then it is dense in $M$
%\item \label{decmpositionofdomain}$D(T)=M\oplus \big(D(T)\cap M^{\bot}\big)$.
%\end{enumerate}
%\end{lemma}
%\begin{proof}
%Proof of (\ref{denseness}):
%Let $x\in M$ and $(x_n)\subseteq D(T)$ be such that $x_n\rightarrow x$ as $n\rightarrow \infty$. Let $x_n=u_n+v_n$ for all $n$, where $u_n\in M$ and $v_n\in M^{\bot}$. Since $(x_n)$ is convergent, by Pythagorean property, $u_n$ and $v_n$ converges, say to $u\in M$ and $v\in M^{\bot}$, respectively. Then $x=u+v$. Since $x\in M$, it follows that $v=x-u\in M$, concluding $v=0$. Hence $x=u$. Also, note that $u_n=x_n-v_n\in D(T)$. This completes the proof.
%
%Proof of (\ref{decmpositionofdomain}):  Let $x\in M$. Then $x=u+v$ for some $u\in M$ and $v\in M^\bot$. Since $x,u\in D(T)$, it follows that $v\in D(T)$. Hence the conclusion follows.
%\end{proof}

%\begin{rmk}
%Let $M$ and $T$ be as in  Lemma \ref{decompositionwithclosedsubsp}. Then replacing $M$ by $M^{\bot}$, we can conclude that if $D(T)\cap M$ is non zero, then it is  dense in $M$.
%
%
%\end{rmk}

\begin{defn}
Let $T\in \mathcal C(H_1,H_2)$ be densely defined. Then $T$ is called absolutely minimum attaining operator if $T|_M:D(T)\cap M\rightarrow H_2$ is minimum attaining for each non-
zero closed subspace $M$ of $H_1$. In other words,  $T$ is absolutely minimum attaining if  there exists $x_0\in   D\big(T|_M\big)$ with $\|x_0\|=1$ such that $\|Tx_0\|=m(T|_M)$.
\end{defn}

Note that  if $T\in \mathcal C(H_1,H_2)$ is densely defined and $M$ is a closed subspace of $H$, then  the restriction operator $T|_M: D(T)\cap M\rightarrow H_2$  is a closed operator and it is densely defined as $D(T|_M)$ is dense in the Hilbert space $\overline{D(T|_M)}$.

We denote the set of all absolutely minimum attaining operators between $H_1$ and $H_2$ by $\mathcal {AM}_{c}(H_1,H_2)$ and in case if $H_1=H_2=H$, this is denoted by $\mathcal {AM}_{c}(H)$.
This class of operators was  introduced and studied in detail by Carvajal and Neves in \cite{carvajalneves2}. The structure of positive absolutely minimum attaining bounded operators is studied in \cite{grs2}.

\begin{prop}\label{closedrangeproperty}
Let $T\in \mathcal {AM}_{c}(H_1,H_2)$. Then $R(T)$ is closed.
\end{prop}
\begin{proof}
Since $T\in \mathcal {AM}_{c}(H_1,H_2)$, we have $T_0=T|_{N(T)^{\bot}}\in \mathcal {AM}_{c}(N(T)^{\bot},H_2)$ and one-to-one. Hence by Proposition \ref{minmattainingclosedrange}, $R(T_0)$ is closed. It is clear that $R(T)=R(T_0)$.
\end{proof}

\begin{rmk}
 The converse of Proposition \ref{closedrangeproperty} need not be true. Let $P$ be a bounded orthogonal projection with infinite dimensional null space and infinite
 dimensional range space. Then $R(P)$ is closed but $P$ is not absolutely minimum attaining by \cite[Lemma 3.2]{carvajalneves2}.
\end{rmk}

Let $M$ be a closed subspace of $H$ and $T\in \mathcal C(H)$ be densely defined. Then $M$ is said to be invariant under $T$,
if $T(M\cap D(T))\subseteq M$.

Let $P:=P_M$. If $P(D(T))\subseteq D(T)$ and $(I-P)(D(T))\subseteq D(T)$, then
\begin{equation*}
T= \left(
           \begin{array}{cc}
             T_{11} & T_{12} \\
             T_{21} & T_{22} \\
           \end{array}
         \right),
\end{equation*}
where $T_{ij}=P_iTP_j|M_j\; (i,j=1,2)$. Here $P_1=P$ and $P_2=I-P$. It is known that $M$ is invariant under $T$ if and only if $T_{21}=0$. Also, $M$ reduces $T$ if and only if $T_{21}=0=T_{12}$.

%\begin{thm}\label{decompositionofanoperator}\cite[page 287, V.5]{taylorlay}
%Let $T$ be reduced by $M$. Let $T_1=T|_M$ and $T_2=T|_{M^\bot}$. Then the following hold:
%\begin{enumerate}
%\item $D(T)=D(T_1)\oplus D(T_2)$
%\item $N(T)=N(T_1)\oplus N(T_2)$
%\item $R(T)=R(T_1)\oplus R(T_2)$
%\item $T^{-1}$ exists if and only if $T_i^{-1}$ exists for $i=1,2$. In this case, $T^{-1}$ is reduced by $M$ and $T^{-1}|_M=T_1^{-1}$ and $T^{-1}|_{M^\bot}=T_2^{-1}$
%\item $R(T)=H$ if and only if $R(T_1)=M$ and $R(T_2)=M^\bot$.
%\end{enumerate}
%\end{thm}
\begin{rmk} Let $T\in \mathcal{C}(H_1,H_2)$ be densely defined. Assume that $M$ reduces $T$ and $T_1=T|_M$ and $T_2=T|_{M^\bot}$. Then the following can be easily verified:
\begin{enumerate}
\item $m(T)=\text{min}{\{m(T_1), m(T_2)}\}$
\item $T\in \mathcal M_{c}(H)$ if and only if the operator $T_j$ with $m(T_j)=m(T),\; (j=1 \; \text{or}\; 2)$,  is minimum attaining.

\end{enumerate}

\end{rmk}
\begin{lemma}\label{mpiforblocks}
Let $T\in \mathcal C(H)$ be densely defined. If $M$ reduces $T$, then $T^\dagger|_M=(T|_M)^\dagger$.
\end{lemma}
\begin{proof}
Since $M$ is reducing subspace, we have
\begin{equation*}
T=\left(
    \begin{array}{cc}
      T_1 & 0 \\
      0 & T_2 \\
    \end{array}
  \right),
\end{equation*}
where $T_1=T|_M$ and $T_2=T|_{M^\bot}$. Since $R(T)$ is closed, by Theorem \cite[page 287, V.5]{taylorlay}, $R(T_i)$ is closed for $i=1,2$. Let $S=\left(
                                                                                                                                                   \begin{array}{cc}
                                                                                                                                                     T_1^{\dagger} & 0 \\
                                                                                                                                                     0& T_2^{\dagger} \\
                                                                                                                                                   \end{array}
                                                                                                                                                 \right).$
Note that $S\in \mathcal B(H)$ and it can be verified that $S$ satisfies all the conditions of the Moore-Penrose inverse. Since $T^\dagger$ is unique, it follows that $S=T^\dagger$.
This proves the claim.
\end{proof}
\begin{thm}\label{reciprocityproperty1}
Let $T\in \mathcal C(H)$ be densely defined  and have a bounded inverse.  Let $M$ be a subspace of $H$. Then
\begin{enumerate}
\item $m(T|_M)=\dfrac{1}{\|T^{-1}|_{T(M\cap D(T))}\|}$ \label{minmnormrelationsubsp1}
\item if $M$ is closed, then $T(M\cap D(T))$ is closed \label{imgaeofclosedsubsp}
\item If $N$ is any subspace of $H$, then $m(T|_{T^{-1}(N}))=\dfrac{1}{\|T^{-1}|_{N}\|}$.\\
Furthermore, if $N$ is closed and $T\in \mathcal B(H)$, then $T^{-1}(N)$ is closed. \label{minmnormrelationsubsp2}
\item $T\in \mathcal {AM}_{c}(H)$ if and only if $T^{-1}\in \mathcal {AN}(H)$. \label{dualrelationabsminm}
\end{enumerate}
\end{thm}

\begin{proof}
Proof of (\ref{minmnormrelationsubsp1}): First, note that as $T$ is one-to-one, we have $D(T)=C(T)$ and $m(T)=\gamma(T)$.
By definition, \begin{align*}
m(T|_M)&=\inf{\Big\{\dfrac{\|Tx\|}{\|x\|}:x\in M\cap D(T),\; x\neq 0}\Big\}\\
       &=\dfrac{1}{\sup{\Big\{\dfrac{\|x\|}{\|Tx\|}:x\in M\cap C(T),\; x\neq 0}\Big\}}\\
       &=\dfrac{1}{\sup{\Big\{\dfrac{\|T^{-1}y\|}{\|y\|}:y=Tx\in T(M\cap D(T)),\; x\neq 0}\Big\}}\\
       &=\dfrac{1}{\|T^{-1}|_{T(M\cap D(T))}\|}.
\end{align*}

Proof of (\ref{imgaeofclosedsubsp}):
 Let $N:=T(M\cap D(T))$ and  let $y\in \overline{N} $. Let $(x_n)\subset M\cap D(T)$ be such that
 $y=\displaystyle \lim_{n\rightarrow \infty }Tx_n$. Since $T^{-1}\in \mathcal B(H)$,
 it follows that $\displaystyle\lim_{n\rightarrow \infty}x_n=T^{-1}y$. Since $M$ is closed, we can conclude that
$T^{-1}y\in M\cap D(T)$. Since $R(T)$ is closed, $y\in R(T)$. Hence $y=T(T^{-1}y)\in T(M\cap D(T))$.

Proof of (\ref{minmnormrelationsubsp2}): This goes along the similar lines of (\ref{minmnormrelationsubsp1}) and (\ref{imgaeofclosedsubsp}).

Proof of (\ref{dualrelationabsminm}):
 If $M=H$, then by Theorem \ref{equivalencegeninvcondition}, we have that $T\in \mathcal M_{c}(H)$ if and only if $T^{-1}\in \mathcal N(H)$. Hence assume that ${\{0}\}\neq M\subset H$. Let $T^{-1}\in \mathcal {AN}(H)$.
  Let  $X=T(M\cap D(T))$ and $R_X=T^{-1}|_{X}$. By (\ref{imgaeofclosedsubsp}), $X$ is closed. Since  $R_X\in \mathcal N(X,H)$, there exists $y_0\in S_X$, such that $\|R_Xy_0\|=\|R_X\|$. This is equivalent to the fact that $R_X^*R_Xy_0=\|R_X\|^2y_0$. Let $y_0=Tx_0$ for some $x_0\in M\cap D(T)$, we get  $R_X^*x_0=\|R_X\|^2Tx_0$. Therefore $\|Tx_0\|=\dfrac{\|R_X^*x_0\|}{\|R_X\|^2}\leq \dfrac{\|x_0\|}{\|R_X\|}=m(T|_M)\, \|x_0\|$ by (\ref{minmnormrelationsubsp1}). Writing $z_0=\dfrac{x_0}{\|x_0\|}$, we get that $\|Tz_0\|\leq m(T|_M)$. But the other inequality holds clearly. Hence $T|_M\in \mathcal M_{c}(M,H)$.

  Conversely, assume that $T\in \mathcal {AM}_{c}(H)$. Let $N$ be a closed subspace of $H$ and let $M:=T^{-1}(N)\subseteq C(T)$.  Since $T\in \mathcal {AM}_{c}(H)$, we have $T_M:=T|_M\in \mathcal{M}_{c}(M,H)$. It can be easily verified that $T_M$ is closed, since $T$ is closed. Since $D(T)\cap M$ is dense in $\overline{D(T)\cap M}$, $T_M$ is densely defined operator. Hence $T_M^*:D(T_M^*)\rightarrow \overline{M}$  exists. By Theorem \ref{equivalentwithgramoperator} and Proposition \ref{eigenvalueforpositive},  there exists  $x_0\in S_{D(T_M^*T_M)}$ such that
\begin{equation}\label{restrictedeigenvalue}
T_M^*T_Mx_0=m(T_M)^2x_0.
\end{equation}
As $T$ is bounded below, $T_M$ is bounded below and hence $m(T_M)>0$.   Let $x_0=T^{-1 }y_0$  for some $y_0\in N$. Then  Equation \ref{restrictedeigenvalue} takes the form:
\begin{equation}\label{restrictedeigenvalue2}
T_M^*y_0=m(T_M)^2T^{-1}y_0.
\end{equation}
First, observe that $R(T_M)=T_M(M\cap D(T))=T(D(T)\cap M)=T(M)=T(T^{-1}(N))=N$.  Hence  $y_0\in N=R(T_M)=N((T_M)^*)^{\bot}$. Taking norm both sides of Equation \ref{restrictedeigenvalue2},  we get
\begin{equation*}
\|T^{-1}{y_0\|}=\dfrac{\|T_M^{*}y_0\|}{m(T_M)^2}\geq \dfrac{\gamma(T_M^*)\|y_0\|}{m(T_M)^2}=\dfrac{\gamma(T_M)\|y_0\|}{m(T_M)^2}\geq \dfrac{m(T_M)\|y_0\|}{m(T_M)^2}=\dfrac{\|y_0\|}{m(T_M)}.
\end{equation*}
Hence $z_0=\dfrac{y_0}{\|y_0\|}\in S_{N}$ and $\|T^{-1}(z_0)\|=\|T^{-1}|_N\|$.
\end{proof}

\begin{cor}
 Let $T\in \mathcal B(H)$ be such that $T^{-1}\in \mathcal B(H)$. Then $T\in \mathcal {AM}_{c}(H)$ if and only if $T^{-1}\in \mathcal{AN}(H)$.
\end{cor}

\begin{thm}\label{structurethm}
Let $T\in \mathcal {AM}_{c}(H)$, positive and not bounded.  Assume that $T$ is one-to-one. Then there exists an unbounded (increasing) sequence ${\{\lambda_n}\}$ of  eigenvalues of $T$ with corresponding eigenvectors ${\{\phi_n}\}$ such that
\begin{enumerate}
\item
\begin{align*}
D(T)&= {\{x\in H: \displaystyle \sum_{n=1}^{\infty} \lambda_n^2\,|\langle x,\phi_n\rangle|^2<\infty}\} \; and\\
Tx&=\sum_{n=1}^\infty\lambda_n\langle x,\phi_n\rangle \phi_n, \; \text{for all}\; x\in D(T).
\end{align*}
The series in the above representation converges in the  strong operator topology. Moreover, $T^{-1}$ is compact.  \label{repn}
\item  $\sigma(T)={\{\lambda_n:n\in \mathbb N}\}$=$\sigma_{p}(T)$  \label{spectrumcalculation}
\item if $\mu \in \sigma_p(T)$,  then $\mu$ is an eigenvalue with  finite multiplicity \label{finitemultiplicity}
\item  $\overline{\text{span}}{\{\phi_n:n\in \mathbb N}\}=H$. \label{ONBeigenvectors}
\end{enumerate}
\end{thm}
\begin{proof}

Proof of (\ref{repn}):
First note that as $T$ is one-to-one and $R(T)$ is closed, $T$ is bounded below. Since $T\geq 0$, $T^{-1}$ exists and bounded.  If $T\in \mathcal {AM}_{c}(H)$, then $T^{-1}\in \mathcal{AN}(H)$ by Theorem \ref{reciprocityproperty1}. Hence by \cite[Theorem 2.5]{rameshvenku}, there exists unique triple $(K,F,\alpha)$, where $K\in \mathcal K(H)$ is positive, $F\in \mathcal F(H)$ positive and $\alpha \geq 0$ such that $KF=0=FK,\; F\leq \alpha I $ and $T^{-1}=\alpha I+K-F$. If $\alpha=0$, then $F=0$ and hence $T^{-1}=K\in \mathcal K(H)$. Next, assume that $\alpha>0$. In this case, $R(T^{-1})=D(T)$ is closed by \cite[Proposition 2.8]{rameshvenku} is closed. Since $T$ is densely defined, we must have that $D(T)=H$. By the closed graph theorem $T$ must be bounded, a contradiction. Hence $\alpha>0$ is not possible. This implies that $\alpha=0$ and hence $T^{-1}\in \mathcal K(H)$.

By the spectral theorem, there exists increasing sequence $(\mu_n)$ of positive eigenvalues of $T^{-1}$ with corresponding eigenvectors ${\{\phi_n:n\in \mathbb N}\}$ such that
\begin{equation}\label{inversecptrepn}
T^{-1}y=\displaystyle \sum_{n=1}^{\infty}\mu_n\langle y,\phi_n\rangle \phi_n,\; \text{for all}\; y\in H.
\end{equation}
The sequence $\mu_n\rightarrow 0$ as $n\rightarrow \infty$. More over, the above series converges to $T^{-1}$ in the operator norm of $\mathcal B(H)$. We can also observe that the sequence $(\mu_n)$ is an infinite sequence. Otherwise, $T^{-1}$ is a finite rank operator and $\sigma(T)$ is bounded. By \cite{knr} this implies that $T$ is bounded which leads to a contradiction. Also, since $T^{-1}$ is compact, $\sigma(T^{-1})={\{\mu_n:n\in \mathbb N}\}=\sigma_p(T)$ and each $\mu_n$ has finite multiplicity. Also, $\mu_{n+1}\leq \mu_{n}$ for each $n\in \mathbb N$.

%Hence by \cite[Theorem 3.8]{SP}, we have  $T^{-1}v_{\alpha}=\beta_{\alpha}v_{\alpha}, \;\alpha \in \Lambda$ . In fact, ${\{\beta_{\alpha}:\alpha \in \Lambda}\}$ is countable (multiplicities of $\beta_{\alpha}$ taken into account).  There exists at most one eigenvalue with infinite multiplicity, say $\beta_{\alpha_0}$.
%
%In this case, ${\{ \beta_{\alpha}:\alpha \in \Lambda}\} $ is a finite set. This means $\sigma(T^{-1})=\overline{{\{\beta_{\alpha}:\alpha\in \Lambda}\}}$  is finite and hence  $\sigma(T)=\overline{{\{\beta_{\alpha}^{-1}:\alpha \in \Lambda}\}}$ is finite set and hence bounded. Since $T\geq 0$, $T$ must be bounded by \cite[Theorem 4.6]{knr}. This is a contradiction to our assumption. Hence every eigenvalue of $T^{-1}$ has finite multiplicity and the same is true for $T$.
%So we can replace $\Lambda$ by $\mathbb N$  and the net ${\{\beta_{\alpha}}\}_{\alpha \in \Lambda}$ by the sequence ${\{\beta_n}\}_{n\in \mathbb N}$. Since ${\{\beta_n}\}_{n\in \mathbb N}$ is a decreasing sequence, it may have a unique limit point, say $\beta$. We claim that $\beta=0$. If this is not the case, $\beta_n^{-1}\rightarrow \beta^{-1}$. Hence $\sigma(T)={\{\beta_n^{-1}:n\in \mathbb N}\}\cup{\{\beta^{-1}}\}$ is a bounded set and $T$ is bounded by the earlier argument. Thus $\beta=0$.
%From, the above discussion, we can conclude that ${\{\beta_{n}}\}$  is a decreasing sequence converging to zero and each eigenvalue $\beta_n$ is of  finite multiplicity, $T^{-1}$ must be compact by \cite[ page 181]{ggk1}.

Let $\lambda_n:=\mu_n^{-1}$ for all $n\in \mathbb N$.
As $T^{-1}$ is compact, by \cite[Theorem 6.1, page 214]{ggk1}, it follows that
\begin{align*}
D(T)&={\{x\in H: \displaystyle \sum_{n=1}^{\infty} \lambda_n^2\,|\langle x,\phi_n\rangle|^2<\infty}\} \; and \\
Tx&=\displaystyle \sum_{n=1}^{\infty} \lambda_n\langle x,\phi_n\rangle \phi_n, \; \text{for all}\; x\in D(T).
\end{align*}

On the other hand, if $T^{-1}$  is compact, by Theorem \ref{reciprocityproperty1}, $T\in \mathcal {AM}_{c}(H)$.

Proof of  statement (\ref{spectrumcalculation}) is clear.       The statement  (\ref{finitemultiplicity}) is proved in (\ref{repn}).

Proof of (\ref{ONBeigenvectors}): Since $T^{-1}$ is compact, $R(T^{-1})=D(T)$ is separable and by the representation above, we have that $H=\overline{D(T)}=\overline{R(T^{-1})}=\overline{\text{span}}{\{v_n:n\in \mathbb N}\}$.
 \end{proof}
 \begin{rmk}
  If $T\in \mathcal B(H)$, then the conclusion  (\ref{repn}) of Theorem \ref{structurethm} is not  true. The unboundedness of the operator is used to get the inverse to be compact.
 \end{rmk}
\begin{thm}\label{characterization1}
Let $T\in \mathcal C(H)$ be densely defined  and one-to-one but not bounded. Then
\begin{enumerate}
\item $T\in \mathcal {AM}_{c}(H)$ if and only if $T^*T\in \mathcal  {AM}_{c}(H)$\label{gram}
\item $T\in \mathcal {AM}_{c}(H)$ if and only if $T^{\dagger}\in \mathcal  {K}(H)$. \label{geninvcompact}
\end{enumerate}
\end{thm}
\begin{proof}

Proof of (\ref{gram}):  If $T\in \mathcal {AM}_{c}(H)$, then $R(T)$ is closed. As $T$ is one-to-one,  $T$ is bounded below. Also, since $|T|$ and $T^*T$ are bounded below and positive,  both have bounded inverse. Hence
\begin{align*}
T\in \mathcal {AM}_{c}(H) \Leftrightarrow |T|\in  \mathcal {AM}_{c}(H) &\Leftrightarrow |T|^{-1}\in  \mathcal{K}(H)  \; ( \text{by Theorem \ref{structurethm}})\\
                                     &\Leftrightarrow |T|^{-2}\in  \mathcal{K}(H)\\
                                     &\Leftrightarrow (T^*T)^{-1}\in  \mathcal{K}(H)\\
                                      &\Leftrightarrow T^*T\in  \mathcal {AM}_{c}(H)  \; ( \text{by Theorem \ref{structurethm}}).
\end{align*}
Proof of (\ref{geninvcompact}):
By  (\ref{gram}),
\begin{align*}
T\in \mathcal {AM}_{c}(H)& \Leftrightarrow  |T|\in \mathcal {AM}_{c}(H)
                                    &\Leftrightarrow  |T|^{-1} \in \mathcal{K}(H)  \; ( \text{by Theorem \ref{structurethm}})\\
                                    &\Rightarrow  |T|^{\dagger}\in \mathcal{K}(H).
                                     \end{align*}

On the other hand, if $T^{\dagger}\in \mathcal K(H)$, then $R(T)$  is closed. As $T$ is one-to-one,  $T$ must be bounded below. This implies that $|T|^{-1}\in \mathcal B(H)$.
Thus,

\begin{align*}
T^{\dagger} \in \mathcal K(H) \Leftrightarrow (T^{*})^{\dagger}\in \mathcal K(H)
                                               &\Leftrightarrow |(T^*)^{\dagger}|\in \mathcal K(H)\\
                                              &\Leftrightarrow |T|^{\dagger}=|T|^{-1}\in \mathcal K(H) \; (\text{by (\ref{mprelations2}) of Proposition \ref{mprelations} })\\
                                              &\Leftrightarrow |T|\in \mathcal{AM}_{c}(H)\\
                                              &\Leftrightarrow T\in \mathcal {AM}_{c}(H). \qedhere
\end{align*}
\end{proof}

\begin{thm}\label{dualwithadjoint}
Let $T\in \mathcal C(H)$ be densely defined unbounded and have a bounded inverse. Then $T\in \mathcal {AM}_{c}(H)$ if and only if $T^*\in \mathcal  {AM}_{c}(H)$.
\end{thm}
\begin{proof}
First observe that both $T$ and $T^*$ are bounded below. We know that $T\in \mathcal {AM}_{c}(H)$ if and only if $T^{-1}\in \mathcal K(H)$. This is true if and only if $(T^*)^{-1}\in \mathcal K(H)$. Now, by \ref{geninvcompact} of Theorem \ref{characterization1}, this is equivalent to the fact that $T^*\in \mathcal {AM}_{c}(H)$.
\end{proof}
\begin{defn}[Hyper invariant subspace]Let $T\in \mathcal C(H)$ be densely defined and $M$ be a closed subspace of $H$. Then $M$ is said to be hyperinvariant subspace of $T$ if $M$ is invariant under every $S\in \mathcal B(H)$ such that $ST\subseteq TS$.
\end{defn}
\begin{thm}(Lomonosov )\label{Lomonosovthm}\cite{radjavirosenthal}
Every operator that commutes with a non-zero compact operator and is not a multiple of the identity has a non-trivial hyperinvariant subspace.
\end{thm}

Using Theorem \ref{Lomonosovthm} we will prove that every $\mathcal{AM}$-operator has a non trivial hyperinvariant subspace.
\begin{thm}
Let $T\in \mathcal {AM}_{c}(H)$, unbounded and  $T^{-1}\in \mathcal B(H)$. Then $T$ has a non trivial hyper invariant subspace.
\end{thm}
\begin{proof}
Let $S\in \mathcal B(H)$ be such that $ST\subseteq TS$. That is $STx=TSx$ for all $x\in D(T)$. Then it can be easily verified that $T^{-1}S=ST^{-1}$. But $T^{-1}\in \mathcal K(H)$ by (\ref{geninvcompact}) of Theorem  \ref{characterization1}. Now,  by Theorem \ref{Lomonosovthm},  $T^{-1}$  has a non trivial invariant subspace, say $M$. Then $M$ is invariant under $S$.  Thus the conclusion follows.
\end{proof}

Now, we can drop the  condition that the operator to be one-to-one  in Theorem \ref{structurethm} and prove the result.

\begin{thm}\label{mpiseparablerange}
Let $T\in \mathcal {AM}_{c}(H)$ be, positive but not bounded. Then
\begin{enumerate}
\item $T^{\dagger}$  is compact
\item $R(T)$ is separable.
\end{enumerate}
\end{thm}
\begin{proof}
 Since $N(T)$ reduces  $T$, we can write $T=\left(
                                              \begin{array}{cc}
                                                T_0 & 0 \\
                                                0 & T_1\\
                                              \end{array}
                                            \right)
 $,
 where $T_0=T|_{N(T)}$ and $T_1=T_{N(T)^{\bot}}$. Then  by Lemma \ref{mpiforblocks},  $T^{\dagger}=\left(
                                                                                                     \begin{array}{cc}
                                                                                                       T_0^{\dagger} & 0 \\
                                                                                                       0 & T_1^{-1} \\
                                                                                                     \end{array}
                                                                                                   \right)
 $. As $T_1\in \mathcal{AM}(N(T)^{\bot})$, by Theorem \ref{characterization1}, $T_1^{-1}$ is compact.
 Note that $T_0=0$ if $N(T)\neq {\{0}\}$ and $T=T_1$ if $N(T)={\{0}\}$. Hence $T^{\dagger}$ is compact.
 Also $R(T_1)$ is separable by  (\ref{ONBeigenvectors}) of Theorem \ref{structurethm}. Now the conclusion follows as $R(T)=R(T_1)$.
\end{proof}

Using Theorem \ref{mpiseparablerange}, we can prove a more general result.

\begin{thm}\label{MPIcompact}
Let $T\in \mathcal {AM}_{c}(H)$, but not bounded. Then
\begin{enumerate}
\item $T^{\dagger}$  is compact \label{compactMPI-gen}
\item  $N(T)^{\bot}$ and $R(T)$ are separable. \label{separablerange-gen}
\end{enumerate}
\end{thm}
\begin{proof}
Proof of (\ref{compactMPI-gen}): We have $T\in \mathcal {AM}_{c}(H)$ if and only if  $|T|\in \mathcal {AM}_{c}(H)$. Hence $|T|^{\dagger}\in \mathcal K(H)$ by Theorem \ref{mpiseparablerange}. But, by Proposition \ref{mprelations}, $|T|^{\dagger}=|(T^{\dagger})^*|$ and hence $(T^{\dagger})^*\in \mathcal{K}(H)$. This implies  that $T^{\dagger}\in \mathcal{K}(H)$.

Proof of (\ref{separablerange-gen}): Since $T^{\dagger}$ is compact, $R(T^{\dagger})=C(T)$ is separable. Hence $N(T)^{\bot}$ is separable.  Since $R(T)$ is closed,  $R(T^*)$ must be closed and since $N(T)^{\bot}=R(T^*)$ , $R(T^*)$ is separable.
But, $R(T^*)$ is separable if and only if $R(T)$ is separable by \cite[Problem 11.4.6, page 362]{aliprantisabramovich}.
\end{proof}
\begin{qn}
If $T\in \mathcal C(H)$ is densely defined and  $T^{\dagger}\in \mathcal K(H)$. Is it true that $T\in \mathcal {AM}_{c}(H)$.
\end{qn}

%\begin{eg}[Hermite Operator]
%\end{eg}
%\begin{eg}[Multiplication Operator ]
%\end{eg}
\begin{thm}\label{self-adjointAM}
Let $T$ be a densely defined and self-adjoint, one-to-one operator on an infinite dimensional  Hilbert space $H$ which is not bounded. Then the following are equivalent:
\begin{enumerate}
\item \label{AMoperator} $T\in \mathcal {AM}_{c}(H)$
\item \label{compactinverse}$T^{-1}\in \mathcal{K}(H)$
\item\label{ONBcriteria} there exists a real sequence $(\lambda_n)$ and an orthonormal basis ${\{v_n:n\in \mathbb N}\}$   of $H$ such that $\displaystyle \lim_{n\rightarrow \infty}|\lambda_n|=\infty$ and $Tv_n=\lambda_nv_n$ for each $n\in \mathbb N$
\item \label{discretespectcriteria} $T$ has purely discrete spectrum
\item\label{resolventcriteria} the resolvent $R_{\lambda}(T):=(T-\lambda I)^{-1}$ is compact for one, and hence for all $\lambda\in \rho(T)$
\item\label{embeddingmapcriteria} the embedding map $J_T: (D(T),\|\cdot \|_T)\rightarrow H$ is compact (here $\|x\|_T=\big(\|x\|^2+\|Tx\|^2\big)^{\frac{1}{2}},\; x\in D(T)$).
\end{enumerate}
\end{thm}
\begin{proof}
If $T\in \mathcal {AM}_{c}(H)$, then $R(T)$ is closed. As $T$ is one-to-one, $T$ must be bounded below and hence $T^{-1}\in \mathcal B(H)$. Now, by (\ref{geninvcompact}) of Theorem \ref{characterization1}, $T^{-1}\in \mathcal K(H)$. Again
by (\ref{geninvcompact}) of Theorem \ref{characterization1}, if $T^{-1}\in \mathcal K(H)$, then $T\in \mathcal {AM}_{c}(H)$. Thus (\ref{AMoperator}) and (\ref{compactinverse}) are equivalent. The equivalence of (\ref{compactinverse})-(\ref{embeddingmapcriteria}) follows by \cite[Proposition 5.12, page 94]{schmudgen}.
\end{proof}
Next, we give an example of $\mathcal{AM}$-operator.
\begin{eg}\label{sturmliouville}
Let  $p,p^{'},q,w$ be continuous real valued functions defined on
$[a,b]$ with $a<b$ and $w(t)>0$ for all $t\in[a,b]$. Consider the real Hilbert space
$$H:=\left\{u\colon \int_a^b|u|^2w<\infty
 \right\}$$ with the inner product
 \begin{equation*}
 \langle u,v\rangle:=\int_a^bu(x)v(x)w(x)dx.
 \end{equation*}
 Let $L$ be the \textbf{Sturm-Liouville operator} given by
 \begin{equation*}
 Lu:=\frac{1}{w}[-(pu^{'})^{'}+qu]
 \end{equation*}
 with
 \begin{equation*}
 D(L)=\left\{u\in H\colon u\in C^2[a,b], \begin{split}\beta_1u(a)+\gamma_1u^{'}(a)&=0,\\
 \beta_2u(b)+\gamma_2u^{'}(b)&=0,\\|\beta_1|+|\gamma_1|&>0,\\|\beta_2|+|\gamma_2|&>0\end{split} \right\}.
 \end{equation*}
 Since  $D(L)$ contains continuous functions defined on $[a,b]$ with compact support, $L$ is densely defined operator.
 Also $L$ is symmetric (See \cite[Chapter 7, section 5]{naysell}. Let us assume that $0\notin
 \sigma_p(L)$. In this case it easy to see that $L^{-1}$ is
 compact and self adjoint. Let $B:={\{\phi_1,\phi_2,\dots }\}$ is an  orthonormal basis for $H$ such that $Lv_n=\mu_n v_n$,
where $\mu_1,\mu_2,\dots$ is a sequence of real numbers which are
eigenvalues of $L$. In this case every $u\in H$ can be expressed
as $$u=\displaystyle \sum_{n=1}^\infty \langle u,v_n\rangle v_n.$$ If
$u\in D(L)$, then
$$\displaystyle \sum_{n=1}^\infty |\langle u,v_n\rangle|^2 \mu_n^2 <\infty$$
and $$Lu=\displaystyle \sum_{n=1}^\infty \mu_n\langle u,v_n\rangle v_n.$$

Note that $L^{-1}y=\displaystyle \sum_{n=1}^\infty \mu_n^{-1}\langle y ,v_n\rangle v_n$ for all $y\in H$. It  is clear that $|L^{-1}|$ is compact and by Proposition \ref{mprelations}, we have $|L|^{-1}=|L^{-1}|$. Hence by Theorem \ref{self-adjointAM}, $L\in \mathcal {AM}_{c}(H)$.

\end{eg}

We end up this section with the following question:

\begin{qn}
Does every bounded absolutely minimum attaining operator have  a non trivial (hyper)  invariant subspace?
\end{qn}
%\section{Appendix}
%\begin{lemma}\label{cauchytypeinequality}
%Let $T\in \mathcal C(H_1,H_2)$. Then
%\begin{equation*}
%  \|Tx\|^2 \geq m(T)|\langle Tx,x\rangle |,\quad \text{for all}\; x\in D(T).
%  \end{equation*}
%\end{lemma}
%\begin{proof}
%By the Cauchy-Schwartz inequality, we have $|\langle Tx,m(T)x\rangle|\leq m(T)\|Tx\|\|x\|$ and since $m(T)\,\|x\|\leq \|Tx\|$ for all $x\in D(T)$, the result follows.
%\end{proof}
%\begin{rmk}
%For a bounded positive operator Lemma \ref{cauchytypeinequality}, is proved in \cite{carvajalneves2}
%\end{rmk}

\textbf{Author contributions} All authors contributed equally and significantly in this paper. All authors read and approved the final manuscript. 

\begin{center}
\noindent \textbf{Compliance with ethical standards}
\end{center}
\textbf{Conflict of interest} The authors declare that they have no conflict of interest.

\noindent \textbf{Ethical approval}: This article does not contain any studies with human participants or animals performed by any of the authors.

%\bibliographystyle{amsplain}
%\bibliography{rameshmain}
\end{document}